\documentclass{article}

\usepackage{microtype}
\usepackage{graphicx}
\usepackage{subfigure}
\usepackage{booktabs} 

\usepackage{hyperref}

\usepackage[accepted]{icml2024}

\usepackage{amsmath}
\usepackage{amssymb}
\usepackage{mathtools}
\usepackage{amsthm}

\usepackage[capitalize,noabbrev]{cleveref}

\theoremstyle{plain}
\newtheorem{theorem}{Theorem}[section]
\newtheorem{proposition}[theorem]{Proposition}

\theoremstyle{definition}

\theoremstyle{remark}
\newtheorem{remark}[theorem]{Remark}

\usepackage[textsize=tiny]{todonotes}

\icmltitlerunning{Gauss-Newton Natural Gradient Descent for Physics-Informed Computational Fluid Dynamics}

\begin{document}

\twocolumn[
\icmltitle{Gauss-Newton Natural Gradient Descent for Physics-Informed \\ Computational Fluid Dynamics}



\icmlsetsymbol{equal}{*}

\begin{icmlauthorlist}
\icmlauthor{Anas Jnini}{yyy}
\icmlauthor{Flavio Vella}{yyy}
\icmlauthor{Marius Zeinhofer}{comp}
\end{icmlauthorlist}

\icmlaffiliation{yyy}{Dept. of Information Engineering and Computer Science,  University of Trento, Italy}
\icmlaffiliation{comp}{Simula Research Laboratory, Oslo, Norway}

\icmlcorrespondingauthor{Anas Jnini}{anas.jnini@unitn.it}
\icmlcorrespondingauthor{Flavio Vella}{flavio.vella@unitn.it}
\icmlcorrespondingauthor{Marius Zeinhofer}{mariusz@simula.no}

\icmlkeywords{Machine Learning, ICML}

\vskip 0.3in
]



\printAffiliationsAndNotice{\icmlEqualContribution} 

\begin{abstract}
We propose Gauss-Newton's method in function space for the solution of the Navier-Stokes equations in the physics-informed neural network (PINN) framework. Upon discretization, this yields a natural gradient method that provably mimics the function space dynamics. Our computational results demonstrate close to single-precision accuracy measured in relative $L^2$ norm on a number of benchmark problems. To the best of our knowledge, this constitutes the first contribution in the PINN literature that solves the Navier-Stokes equations to this degree of accuracy. Finally, we show that given a suitable integral discretization, the proposed optimization algorithm agrees with Gauss-Newton's method in parameter space. This allows a matrix-free formulation enabling efficient scalability to large network sizes.
\end{abstract}

\section{Introduction}
\label{sec:intro}

\paragraph{Physics-Informed Neural Networks (PINNs)} PINNs are a machine learning tool to solve forward and inverse problems involving partial differential equations (PDEs) using a neural network ansatz. They have been proposed as early as \cite{dissanayake1994neural} and were later popularized by the works \cite{raissi2019physics, karniadakis2021physics}. PINNs are a meshfree method designed for the seamless integration of data and physics. Applications include fluid dynamics \cite{cai2021physics}, solid mechanics \cite{haghighat2021physics} and high-dimensional PDEs \cite{hu2023tackling} to name but a few areas of ongoing research.

\begin{figure}[h]
    \centering
    \begin{tikzpicture}
        \node[inner sep=0pt] (r1) at (0,0)
    {\includegraphics[width=\linewidth]{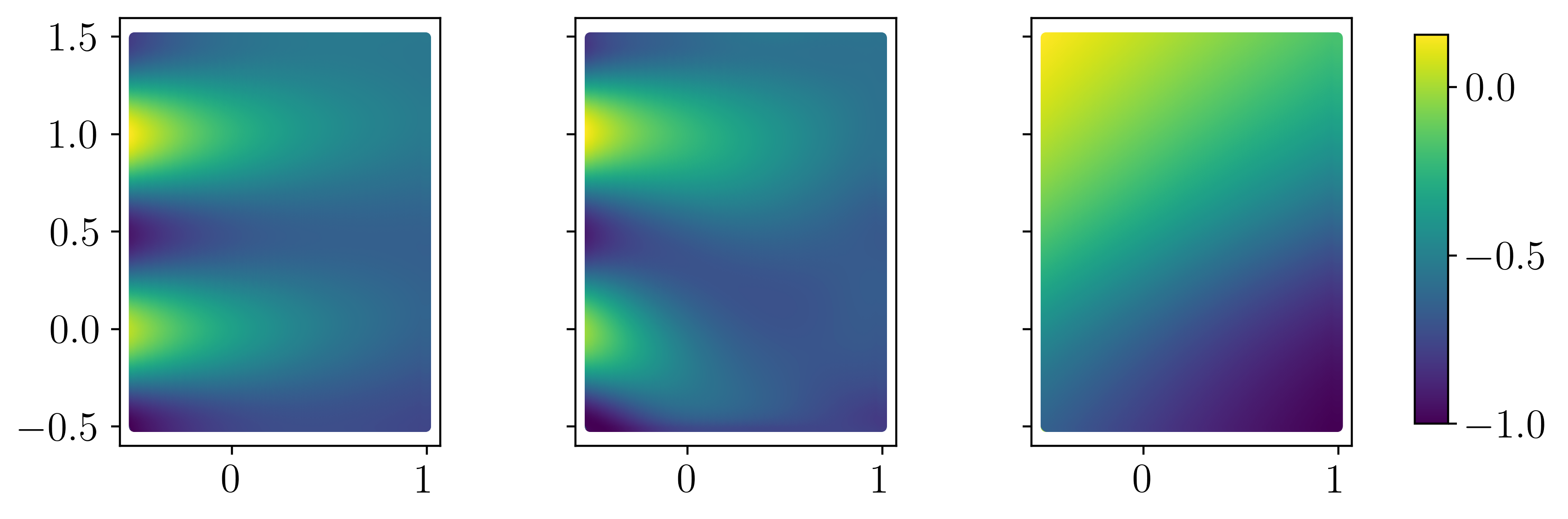}};
        \node[inner sep=0pt] (r1) at (-2.6,1.5)
    {\small 
        \footnotesize $u_\theta-u^\ast$
    };
        \node[inner sep=0pt] (r1) at (-0.3,1.5)
    {
    \footnotesize Gauss-Newton  
    };
        \node[inner sep=0pt] (r1) at (2.1,1.5)
    {\footnotesize Gradient 
    };
    \end{tikzpicture}
    
    \caption{Shown are the error $u_\theta - u^*$ and the push forwards of the Gauss-Newton and Euclidean gradient for the first component of the Kovasznay flow example in Section~\ref{sec:kovasznay}. Note that the error $u_\theta-u^*$ is the optimal update direction and is closely matched by the Gauss-Newton direction. All plots are normed to lie in $[-1,1]$.}\label{fig:kovasznay_pushs}
\end{figure}
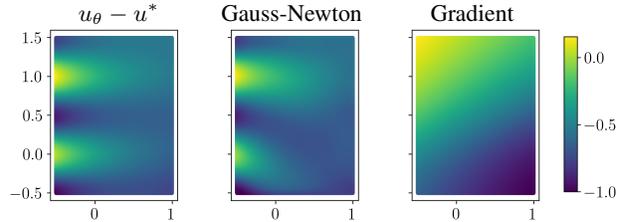

\paragraph{PINN Optimization} The optimization of PINNs is well-known to be challenging beyond classical supervised learning that they resemble most \cite{krishnapriyan2021characterizing, wang2021understanding, zeng2022competitive}. They typically suffer from long training times and mediocre accuracy and errors below $10^{-5}$ in relative $L^2$ norm are rarely observed. A few recent studies, such as those by \cite{muller2023achieving, zeng2022competitive, wang2023multi}, have focused on addressing accuracy in physics-informed neural network. However, to our knowledge, there is a lack of reported work in the literature that demonstrates highly accurate PINN solutions for nonlinear PDEs. The aforementioned studies address linear PDEs efficiently. Achieving high accuracy is crucial in many scientific domains.

\paragraph{The Navier-Stokes Equations} The incompressible Navier-Stokes equations are a coupled system of PDEs that describe the motion of a viscous fluid. They are well-known to be a challenging example of nonlinear coupled partial differential equations. Solving Navier-Stokes equations with PINNs has been intensively investigated in the works \cite{jin2021nsfnets, karniadakis2023solution}, see also the review article \cite{cai2021physics}. To the best of our knowledge, our work is the first contribution that produces highly accurate PINN solutions to the Navier-Stokes equations with errors significantly below $10^{-5}$ measured with respect to the relative $L^2$ norm.

\paragraph{A Function-Space View on Neural Network Training} 
Recently, proposed optimization methods that achieve close to single-precision accuracy for PINNs are based on infinite-dimensional optimization algorithms that are discretized in the tangent space of the neural network ansatz. For instance, Energy Natural Gradient Descent (ENGD) proposed in \cite{muller2023achieving} corresponds to a projected Newton method in \emph{function space}. Competitive PINNs (CPINNs) \cite{zeng2022competitive} can be interpreted -- at least for linear PDEs -- as a Lagrange-Newton method in function space\footnote{This viewpoint differs significantly from the presentation of CPINNs in the original paper and is detailed in \cite{to_appear}.}. In this contribution we follow this paradigm and consider both Newton's method (which leads to ENGD) and Gauss-Newton's method in \emph{function space} and discretize them in the tangent space of the neural network ansatz. As function space algorithms, both methods have the potential of quadratic local convergence and are thus highly effective, at least locally.

The discretization of both Newton and Gauss-Newton's method leads to natural gradient-like second-order methods in parameter space and we address resulting scalability issues by a matrix-free approach, see also Section \ref{sec:matrix_free}.

\paragraph{Main Contributions} Our main contributions can be summarized as follows:
\begin{itemize}
    \item [(i)] We propose a second-order optimization method based on Gauss-Newton's method in function space for the training of physics-informed neural networks. Numerically, the proposed method demonstrates unprecedented accuracy of PINN solutions for the Navier-Stokes equations. See Appendix~\ref{sec:concise_description} for a reference description.

    \item [(ii)] We give an infinite-dimensional differential geometric interpretation, explaining the algorithm's optimization dynamics in \emph{function space}. More precisely, we show that the proposed method follows the function space updates up to an orthogonal projection on the model's tangent space.

    \item [(iii)] We demonstrate that -- given appropriate integral discretization in the PINN formulation -- the proposed method corresponds to the well-known Gauss-Newton method in parameter space. This leads to a matrix-free formulation that allows the applicability of the method to large network sizes.
\end{itemize}
In the following, we will refer to the proposed method as Gauss-Newton Natural Gradient (GNNG) stressing its geometrical interpretation. 

\paragraph{Related Work}
Improving the training for PINNs is an active field of research. A line of resarch concentrates on weighting strategies for the different loss terms corresponding to PDE residuals, data terms, and boundary conditions, we refer to \cite{wang2021understanding}. An alternative method to enhance the training process in physics-informed neural networks is to optimize the selection of collocation points for integral discretization in the loss function's formulation. This typically involves adaptive sampling strategies, which recalibrate collocation points according to error indicators like the PDE residual. For an in-depth understanding of these strategies, we recommend \cite{wu2023comprehensive}. While the contributions above improve the training of PINNs significantly, none of papers report relative $L^2$ errors below $10^{-4}$, not even for benign equations. 

Only recently, several works achieved highly accurate PINN solutions with relative $L^2$ errors below $10^{-5}$. In \cite{muller2023achieving} the authors introduce a natural gradient method that emulates Newton's algorithm in function space. This method demonstrates its efficacy by accurately solving a range of linear PDEs and constitutes the main motivation for the present contribution. Furthermore, in \cite{zeng2022competitive}, the authors reformulate the original PINN formulation into a saddle-point problem and apply competitive gradient descent \cite{schafer2019competitive} for its solution. This approach also reports highly accurate solutions for various linear PDEs. A common element in both methodologies is the adoption of an infinite-dimensional perspective. The natural gradient method in \cite{muller2023achieving} is a discretization of Newton's method in function space and the approach of \cite{zeng2022competitive} can be interpreted as a Lagrange-Newton method, see \cite{to_appear} for details. Furthermore in \cite{siegel2023greedy}, a greedy algorithm for the optimization of shallow neural networks is reported to achieve high-accuracy for a number of PDE problems, both for objective functions in the PINN formulation and in variational form. 

To the best of our knowledge, the present work presents the first contribution in the literature that achieves highly accurate PINN solutions with relative $L^2$ errors below $10^{-5}$ for nonlinear PDEs.

\section{Preliminaries}\label{sec:preliminaries}
\subsection{Notation}
By $\mathcal H, \mathcal Q$ we denote abstract real Hilbert spaces and $\mathcal H^*$ denotes the dual space of $\mathcal H$, i.e, the space of linear continuous functionals $\mathcal H \to \mathbb R$. More generally, the space of continuous linear maps from $\mathcal H$ to $\mathcal Q$ will be denoted by $\mathcal L(\mathcal H, \mathcal Q)$. For $u, v\in\mathcal H$, the inner product in $\mathcal H$ will typically be written as $(u,v)_{\mathcal H}$  and for a functional $f\in \mathcal H^*$ we denote the duality pairing by $f(u) = \langle f, u\rangle_{\mathcal H}$. For a linear bounded map $T:\mathcal H \to \mathcal Q$ we denote its Hilbert space adjoint by $T^*:Q\to\mathcal H$ and it is implicitly defined by the equation
\begin{equation*}
    (Tu, v)_{\mathcal Q} = (u, T^*v)_{\mathcal H}, \quad \text{for all }u\in\mathcal H, v\in\mathcal Q.
\end{equation*}
The Fr\'echet derivative of a map $E:\mathcal H \to \mathcal Q$ is denoted by $DE$ and it is a map
\begin{equation*}
    DE: \mathcal H \to \mathcal L(\mathcal H, \mathcal Q), \quad u\mapsto DE(u).
\end{equation*}
For a reference on general functional analytic concepts, we recommend \cite{zeidler2012applied}. Concrete Hilbert spaces used throughout the manuscript are the space of square-integrable functions, denoted by $L^2(\Omega)$ and the Sobolev space of $s$-times weakly differentiable functions which we denote by $H^s(\Omega)$, see also \cite{adams2003sobolev}. In both cases, $\Omega\subset\mathbb R$ is a bounded, open and connected set where $d=2,3$. 

\subsection{Physics-Informed Neural Networks}
Given a domain $\Omega\subset \mathbb R^d$, a time interval $I=[0,T]$, forcing data $f$, viscosity $\nu>0$, boundary data $g$ and initial data $u_0$, the Navier-Stokes equations are given by
\begin{align}
    \begin{split}\label{eq:navier_stokes}
        \partial_tu - \nu\Delta u + (u\cdot\nabla)u + \nabla p &= f \quad \text{in }  \Omega_T
        \\
        \operatorname{div}u &= 0 \quad \text{in }\Omega_T,
        \\
        u &= g \quad \text{on }I\times\partial\Omega,
        \\
        u(0) &= u_0 \ \   \text{in }\Omega,
    \end{split}
\end{align}
where we assumed Dirichlet information on the boundary $\partial\Omega$ for simplicity of presentation and we abbreviated $I\times\Omega$ by $\Omega_T$. We now convert solving equation \eqref{eq:navier_stokes} into a residual minimization problem. More precisely, solving \eqref{eq:navier_stokes} is equivalent to minimizing 
\begin{align}\label{eq:E_navier_stokes}
    \begin{split}
        E(u,p) 
        &=
        \frac12\|\partial_tu -\nu \Delta u + (u\cdot \nabla)u + \nabla p - f \|^2_{L^2(\Omega_T)}
        \\
        &+
        \frac12 \| \operatorname{div}(u) \|^2_{L^2(\Omega_T)}
        +
        \frac12 \| u - g \|^2_{L^2(I\times\partial\Omega)}
        \\
        &+
        \frac12 \| u(0) - u_0 \|^2_{L^2(\Omega)}.
    \end{split}
\end{align}
over a suitable function space. Employing neural network discretizations $u_\theta$ and $p_\psi$ with parameters and parameter spaces $\theta\in\Theta$ and $\psi\in\Psi$ leads to the PINN formulation of the Navier-Stokes equations. We aim to minimize the loss function $L$ defined as 
\begin{equation}\label{eq:loss_navier_stokes_pinn}
    \min_{\theta, \psi} L(\theta, \psi) = E(u_\theta, p_\psi).
\end{equation}
In practice, the integrals appearing in the definition of $L$ are discretized using Monte Carlo integration or numerical quadrature. More general problem classes including for example observational data or time dependency are handled by including the corresponding terms into the loss function $L$.

\paragraph{Hard-Constraints in PINNs}
Solving the Navier-Stokes equations \eqref{eq:navier_stokes} corresponds to solving three equations simultaneously: The Navier-Stokes equation on $\Omega$, the divergence constraint on $\Omega$ and the boundary conditions on $\partial\Omega$. Both in classical methods and in PINNs it is common to embed the boundary and/or the divergence constraint directly into the ansatz via suitably modifying the network structure. We refer to \cite{richter2022neural} for the imposition of divergence constraints and \cite{sukumar2022exact, lu2021physicsinformed, Dong_2021} for the imposition of boundary values. In the experiment Section below, we explicitly state if and which construction is used.

\section{Function Space Optimization}
Our goal is to design effective and accurate optimizers for the minimization of \eqref{eq:loss_navier_stokes_pinn}. We aim to exploit the structure of $E$ as opposed to the structure of $L$ which is highly non-convex due to the neural network discretization.

\subsection{Second-Order Methods}
To this end, given initial parameters $\theta_0\in\Theta$ and $\psi_0\in\Psi$ we are considering second-order methods of the form
\begin{align}\label{eq:generic_second_order}
    \begin{pmatrix}
        \theta_{k+1} \\
        \psi_{k+1}
    \end{pmatrix}
    =
    \begin{pmatrix}
        \theta_k \\
        \psi_k
    \end{pmatrix}
    -
    \eta_k
    G(\theta_k, \psi_k)^\dagger
    \begin{pmatrix}
        \nabla_\theta L(\theta_k, \psi_k)\\
        \nabla_\psi L(\theta_k, \psi_k)
    \end{pmatrix}
\end{align}
where $\eta_k >0$ corresponds to a suitably chosen step-size. The crucial point is the choice of the matrix $G(\theta_k, \psi_k)$.  It is well known that widely employed methods in the PINN community, such as BFGS $(G \approx \nabla^2L)$ do not yield fully satisfactory results, see for example \cite{muller2023achieving, wang2021understanding} and refer to the numerical experiments presented in this work. We explain our proposed choice for $G$ in the following. Alternatively, the reader may consult Appendix~\ref{sec:concise_description} as a reference.

\subsection{Gauss-Newton in Function Space}
In order to derive a candidate for $G$, we observe that minimizing $E$ is a least squares problem in a Hilbert space setting. Abstractly it possesses the structure
\begin{equation*}
    \min_{w\in\mathcal H} E(w) = \frac12 \| R(w) \|^2_{\mathcal Q},
\end{equation*}
for Hilbert spaces $\mathcal H$, $\mathcal Q$ and a nonlinear operator $R:\mathcal H \to \mathcal Q$. Given an initial value $w_0\in\mathcal H$, Gauss-Newton's method in function space is 
\begin{equation}\label{eq:gauss_newton_f_space}
    w_{k+1}
    =
    w_k
    -[DR(w_k)^*DR(w_k)]^{-1}DE(w_k),
\end{equation}
for $k=0,1,\dots$. Here, we $DR(w_k)^*$ is interpreted as a map $\mathcal Q\to \mathcal H^*$, more precisely, we have
\begin{equation*}
    DR(w_k)^*DR(w_k)\delta_w = (DR(w_k)\delta_w, DR(w_k)(\cdot))_Q
\end{equation*}
as an element of $\mathcal H^*.$ Moreover, we assume sufficient smoothness, and the existence of the inverse in the equation above\footnote{These restrictions can be relaxed, as detailed in \cite{deuflhard1979affine}}. The update in \eqref{eq:gauss_newton_f_space} is motivated by linearizing $R$ around the current iterate and explicitly solving the resulting quadratic minimization problem, see for example \cite{deuflhard1979affine}.

Translating the abstract algorithm to the concrete $E$ resulting from Navier-Stokes equations, we first note that 
\[
    \mathcal H 
    =
    H^1(I,L^2(\Omega))\cap L^2(I,H^2(\Omega)) \times L^2(I,H^1(\Omega))
\] 
and 
\[
    \mathcal Q 
    =
    L^2(\Omega_T)\times L^2(\Omega_T)\times L^2(I\times\partial\Omega)\times L^2(\Omega).
\]
The residual
$R(w) = R(u,p)$ is given by
\begin{align}
\begin{split}
R(u,p) =
\begin{pmatrix}
    \partial_tu - \nu\Delta u + (u\cdot \nabla)u + \nabla p - f
    \\
    \operatorname{div}(u)
    \\
     u - g
     \\
     u(0)-u_0
\end{pmatrix}.
\end{split}
\end{align}
The operator $T = DR(u,p)^*DR(u,p):\mathcal H \to \mathcal H^*$ has block structure
\begin{align}\label{eq:T_full_GN_operator}
    T 
    =
    \begin{pmatrix}
        T_1 & T_2 \\
        T_2^* & T_3
    \end{pmatrix}.
\end{align}
The blocks are operators that map as follows. We abbreviate $N(u,\delta_u) = \partial_t\delta_u-\nu\Delta\delta_u + (\delta_u\cdot\nabla)u + (u\cdot\nabla)\delta_u$ and have
\begin{align}
\begin{split}\label{eq:T_1}
    \langle T_1 \delta_u, \bar \delta_u \rangle
    &=
    (N(u,\bar\delta_u), N(u, \delta_u))_{L^2(\Omega_T)}
    \\
    &+
    (\operatorname{div}(\bar\delta_u), \operatorname{div}(\delta_u))_{L^2(\Omega_T)}
    \\
    &+
    (\bar\delta_u, \delta_u)_{L^2(I\times\partial\Omega)}
    \\
    &+
    (\bar\delta_u(0), \delta_u(0))_{L^2(\Omega)}
    .
\end{split}
\end{align}
For $T_2$ and $T_2^*$ we have the formulas
\begin{align}
    \begin{split}\label{eq:T_2_T_2*}
        \langle T_2 \delta_p, \bar\delta_u \rangle
        &=
        (N(u,\bar\delta_u), \nabla \delta_p)_{L^2(\Omega_T)}
        \\
        \langle T_2^* \delta_u, \bar\delta_p \rangle
        &=
        \langle T_2 \bar\delta_p, \delta_u \rangle.
    \end{split}
\end{align}
Finally, $T_3$ is given by
\begin{align}\label{eq:T_3}
    \langle T_3 \delta_p, \bar\delta_p \rangle = (\nabla \bar\delta_p, \nabla \delta_p)_{L^2(\Omega_T)}.
\end{align}

\subsection{Discretization in Tangent Space}
We are now in the position to specify the matrix $G$ in \eqref{eq:generic_second_order} that corresponds to the Gauss-Newton method in function space. To that end, we use the derivatives of the neural network ansatz with respect to the trainable weights, i.e., 
\begin{equation}\label{eq:tangent_funcs}
    \partial_{\theta_1}u_\theta,\dots,\partial_{\theta_{p_\Theta}} \quad \text{and}\quad \partial_{\psi_1}p_\psi,\dots,\partial_{\psi_{p_\Psi}}
\end{equation}
as arguments to $DR(u,p)^*DR(u,p)$. These functions generate the tangent space of the neural network ansatz 
\begin{equation}\label{eq:tangent_space}
    T_{u_\theta, p_\psi}\mathcal M
    =
    \underset{\substack{i=1,\dots p_\Theta \\ j=1,\dots p_\Psi}}{\operatorname{span}}
    \{ (\partial_{\theta_i}u_\theta,0), (0, \partial_{\psi_j}p_\psi)  \}
\end{equation}
of the neural network ansatz
\begin{equation*}
    \mathcal M = \{(u_\theta, p_\psi)\mid \theta\in \Theta, \psi\in\Psi\}
\end{equation*}
at the parameters $\theta$ and $\psi$ which explains the terminology of the discretization. Following this approach, we obtain a matrix
\begin{align}\label{eq:the_G}
    G^{\text{GN}}(\theta,\psi)
    =
    \begin{pmatrix}
        A & B \\
        B^T & C
    \end{pmatrix}
\end{align}
where $A,B$ and $C$ are block matrices that depend on $\theta$ and $\psi$ and are given in terms of the maps $T_1, T_2$ and $T_3$ by the formulas
\begin{align*}
    A_{ij} 
    &=
    \langle T_1 \partial_{\theta_i}u_\theta, \partial_{\theta_j}u_\theta\rangle,
    \\
    B_{ij}
    &=
    \langle T_2 \partial_{\psi_i}p_\psi, \partial_{\theta_j}u_\theta \rangle,
    \\
    C_{ij}
    &=
    \langle T_3 \partial_{\psi_i}p_\psi, \partial_{\psi_j}p_\psi \rangle.
\end{align*}
Note that for all $\theta\in\Theta$ and $\psi\in\Psi$ the matrix $G^{\text{GN}}(\theta,\psi)$ is positive semi-definite. Up to damping, the matrix $G^{\text{GN}}$ will be our preferred choice in a second order method of the form \eqref{eq:generic_second_order}. 

\begin{remark}[Galerkin Discretization]
    Galerkin discretizations -- well known in the numerical analysis literature, see for instance \cite{brenner2008mathematical} -- refer to the discretization of infinite dimensional operators or bilinear forms using finite dimensional subspaces. In our setting, we can interprete $G^{\text{GN}}$ as a Galerkin discretization of the operator $T$ with respect to the tangent space of the neural network ansatz as generated by the functions \eqref{eq:tangent_funcs} and defined in \eqref{eq:tangent_space}.
\end{remark}

\subsection{Geometrical Interpretation}\label{sec:geometrical_interpretation}
Recall that the optimization dynamics for Gauss-Newton's method in function space are given by
\begin{align}\label{eq:f_space_dynamics}
    \begin{pmatrix}
        u_{{k+1}}
        \\
        p_{{k+1}}
    \end{pmatrix}
    =
    \begin{pmatrix}
        u_{{k}}
        \\
        p_{{k}}
    \end{pmatrix}
    -
    \eta_k
    \begin{pmatrix}
        d^u_k
        \\
        d^p_k
    \end{pmatrix}
\end{align}
where $\eta_k>0$ is a suitable step-size and the additive update direction given via
\begin{align*}
    \begin{pmatrix}
        d^u_k
        \\
        d^p_k
    \end{pmatrix}
    =
    [DR(u_k, p_k)^*DR(u_k, p_k)]^{-1}DE(u_k, p_k),
\end{align*}
we refer also to equation \eqref{eq:gauss_newton_f_space}. We can also interprete these dynamics as a natural gradient method induced by a Riemannian metric
\begin{align}\label{eq:Riemannian_metric_g}
    g(u_k,p_k)((u,p),(v,q)) 
\end{align}
which for $(u,p),(v,q)\in\mathcal H$ is given by the formula
\begin{align}\label{eq:Riemannian_metric_g_formula}
    (DR(u_k,p_k)(u,p), DR(u_k, p_k)(v,q))_{\mathcal Q}.
\end{align}
This means that $g$ is a Riemannian metric on the function space $\mathcal H$ and the neural network ansatz class $\mathcal M$. The iteration \eqref{eq:generic_second_order} using the matrix \eqref{eq:the_G} is thus a natural gradient descent method with the geometry induced by the Riemannian metric $g$ defined above.

\begin{theorem}[Interpretation of Update Direction]\label{thm:update_direction}
    Assume that we employ algorithm \eqref{eq:generic_second_order} using the matrix \eqref{eq:the_G} producing a sequence of neural networks $(u_{\theta_k})$ and $(p_{\psi_k})$. Then it holds
    \begin{align}\label{eq:FS_dynamics_discretized_algorithm}
        \begin{pmatrix}
            u_{\theta_{k+1}} \\
            p_{\psi_{k+1}}
        \end{pmatrix}
        =  
        \begin{pmatrix}
            u_{\theta_{k}} \\
            p_{\psi_{k}}
        \end{pmatrix} 
        -
        \eta_k 
        \Pi_k
        \begin{pmatrix}
            d_{\theta_{k}} \\
            d_{\psi_{k}}
        \end{pmatrix} 
        +
        \epsilon_k,
    \end{align}
    where $\Pi_{k}$ denotes the orthogonal projection onto the tangent space \eqref{eq:tangent_space} with respect to the inner product $g(u_k,p_k)$. The term $\epsilon_k$ corresponds to an error vanishing quadratically in the step and step size length 
    \begin{equation*}
        \epsilon_k 
        =
        \mathcal O(\eta_k^2\lVert G(\theta_k,\psi_k)^\dagger \nabla L(\theta_k, \psi_k) \rVert^2).
    \end{equation*}
\end{theorem}
\begin{proof}
    The proof of can be found in \cite{to_appear} as Theorem 1. However, we need to verify that the Riemannian metric $g$ is positive definite and not merely positive semi-definite. This requires carefully examining regularity properties of the linearized Navier-Stokes equations. The analysis is provided in Appendix \ref{sec:appendix_math}.
\end{proof}
\begin{remark}
    We draw the reader's attention to the close similarity between the function space dynamics \eqref{eq:f_space_dynamics} and the updates of the discretized algorithms, see \eqref{eq:FS_dynamics_discretized_algorithm}.
    This ensures that the high quality\footnote{We refer to the quadratic convergence of the Gauss-Newton method in function space.} direction $(d^u_k, d^p_k)$ is followed as closely as possible.
\end{remark}

\subsection{Newton's Method in Function Space}
We consider Newton's method in function space for minimizing $E$ as given in equation \eqref{eq:E_navier_stokes}. For initial values $(u_0,p_0))$, this leads to the iteration
\begin{align*}
    \begin{pmatrix}
        u_{k+1} \\
        p_{k+1}
    \end{pmatrix}
    =
    \begin{pmatrix}
        u_k \\
        p_k
    \end{pmatrix}
    -
    D^2E(u_k,p_k)^{-1}
    \begin{pmatrix}
        D_uE(u_k,p_k) \\
        D_pE(u_k, p_k)
    \end{pmatrix}
\end{align*}
for $k=0,1,2,\dots$. The second derivative $D^2E(u,p)$ at given functions $u\in H^2(\Omega)$ and $p\in H^1(\Omega)$ differs from $T$ defined in equation \eqref{eq:T_full_GN_operator} only in the first block. More precisely, it is given by
\begin{align*}
    D^2E(u,p)
    =
    \begin{pmatrix}
        T_1 + H & T_2 \\
        T_2^* & T_3
    \end{pmatrix}
\end{align*}
where $T_1, T_2$ and $T_3$ are the operators defined in \eqref{eq:T_1}, \eqref{eq:T_2_T_2*} and \eqref{eq:T_3}. Moreover, $\langle H \delta_u, \bar\delta_u \rangle$ is given by
\begin{align*}
    (-\Delta u + (u\cdot \nabla)u + \nabla p -f, (\delta_u\cdot\nabla) \bar\delta_u + (\bar\delta_u\cdot\nabla)\delta_u))_{\Omega}.
\end{align*}
Discretizing this algorithm leads to the recently proposed energy natural gradient descent \cite{muller2023achieving}. This approach discretizes $D^2E(u,p)$ in the same way as $DR(u,p)^*DR(u,p)$ and yields a matrix $G^{\text{ENGD}}$ which can be used in a second order method of the form \eqref{eq:generic_second_order}. However, in our numerical experiments we observe that this choice leads to a less robust optimization, see also Section~\ref{sec:visual}. 

\begin{remark}
    Interpreting the update direction resulting from the choice $G^{\text{ENGD}}$ is less clear as in the case of $G^{\text{GN}}$. This is due to the possibility of $G^{\text{ENGD}}$ having both positive and negative eigenvalues. In this case, the connection to a natural gradient method is lost.
\end{remark}

\subsection{Connection to Gauss-Newton in Parameter Space}\label{sec:gauss_newton_connection}
We show that Gauss-Newton in parameter space and function space coincide, given a suitable integral discretization. For simplicity we consider the stationary Navier-Stokes equations in this Section, although the results generalize readily. For quadrature points $(x_i)_{i=1,\dots,N}$ in $\Omega$ and $(x^b_i)_{i=1,\dots,N_{\partial\Omega}}$ on $\partial\Omega$ we define the discrete residual $r:(\theta,\psi)\to\mathbb R^{2N_\Omega + N_{\partial\Omega}}$ to be
\begin{align*}
    \begin{pmatrix}
        \frac{1}{\sqrt N}(-\nu\Delta u_\theta + (u_\theta\cdot\nabla)u_\theta + \nabla p_\psi - f)(x_1) \\
        \vdots \\
        \frac{1}{\sqrt N}(-\nu\Delta u_\theta + (u_\theta\cdot\nabla)u_\theta + \nabla p_\psi - f)(x_N) \\
        \frac{1}{\sqrt N}\operatorname{div}(u_\theta)(x_1) \\
        \vdots \\
        \frac{1}{\sqrt N}\operatorname{div}(u_\theta)(x_N) \\
        \frac{1}{\sqrt N_{\partial\Omega}}u_\theta(x_1^b) \\
        \vdots \\
        \frac{1}{\sqrt N_{\partial\Omega}}u_\theta(x_{N_{\partial\Omega}}^b)
    \end{pmatrix}.
\end{align*}
The discretized PINN formulation of \eqref{eq:navier_stokes} reads
\begin{equation}\label{eq:navier_stokes_pinn_discrete}
    \min L(\theta,\psi) = \frac12\|r(\theta,\psi)\|^2_{l^2}.
\end{equation}
It is straight-forward to see that it holds
\begin{equation}\label{eq:Jacobian_transpos_Jacobian}
    G(\theta, \psi)
    =
    J(\theta, \psi)^T \cdot J(\theta, \psi),
\end{equation}
where $J(\theta, \psi)$ denotes the Jacobian of $r$ at $(\theta, \psi)$. In other words, applying Gau\ss-Newton's method to \eqref{eq:navier_stokes_pinn_discrete} agrees with the discretization of the function space algorithm -- if in the discretization of $G$ the same quadrature points are being used.

\subsection{Matrix-Free Formulation}\label{sec:matrix_free}
Formula \eqref{eq:Jacobian_transpos_Jacobian} allows to compute the application of the Gramian $G(\theta, \psi)$ on a vector $v$ in a matrix-free way. This is done using a combination of forward \& backward mode automatic differentiation and requires only constant overhead over a gradient computation of \eqref{eq:navier_stokes_pinn_discrete}, compare also to the discussion in \cite{schraudolph2002fast}. In fact, we compute 
\begin{equation*}
    G(\theta, \psi)v = J^Tw, \quad \text{where } w = Jv.
\end{equation*}
Having access to Gramian-vector products, we can resort to matrix-free solvers for the solution of $G(\theta, \psi)^\dagger\nabla L(\theta,\psi)$. We use the conjugate gradient method, see for instance \cite{trefethen2022numerical}.

\section{Experiments}
\label{sec:Experiments}
We evaluate the GNNG method on three benchmark incompressible Navier-Stokes flows that admit analytical solutions: the two-dimensional steady Kovasznay flow, the two-dimensional unsteady Taylor-Green vortex with periodic boundary conditions, and the three-dimensional unsteady Beltrami flow. We showcase the use of the matrix-free approach on a large neural-network in Appendix~\ref{sec:MFTG}, along with additional resources for the main experiments in Appendix~\ref{sec:Additional}.
\paragraph{Description of the method}
For all our numerical experiments, we realize a GNNG step with a
line search on a logarithmic grid as described in the work by \cite{muller2023achieving}. Depending on the size of the neural-network, we either use a direct solver or a matrix-free approach as described in Section~\ref{sec:matrix_free}.

\paragraph{Evaluation}
We report relative $L^2$ errors for the velocity and pressures which we compute with one order of magnitude more quadrature points to guarantee a faithful representation of the errors.
For brevity, we define the mean component-wise relative \( L^2 \) error. This is computed for each simulation as \( E_{m} = \frac{1}{n} \sum_{i=1}^{n} e_{i} \), where \( e_{i} \) represents the relative \( L^2 \) error for the \( i \)-th component, with \( n \) being the total number of components under consideration. For example, in a system with components \( u \), \( v \), and \( p \), we have \( n = 3 \) and \( e_{i} \in \{ e_{u}, e_{v}, e_{p} \} \). We also report the training loss in Appendix~\ref{sec:Additional} as a measure of the efficiency of the optimizer. We test the performance of GNNG against Adam \cite{kingma2017adam} and the quasi-Newton method BFGS \cite{nocedal1999numerical}. For the optimization with Adam we employ and exponentially decreasing learning rate schedule, starting with an initial learning rate of $10^{-3}$ and decreasing after $1.5 \times 10^{4}$ steps by a factor of $10^{-1}$ every $10^{4}$ steps.

\paragraph{Computation details} 
We employ JAX \cite{jax2018github}, Jaxopt \cite{jaxopt_implicit_diff} and Optax \cite{deepmind2020jax} for our implementation and the optimizers Adam and BFGS, respectively. Direct linear solves of $G(\theta, \psi)^\dagger\nabla L(\theta,\psi)$ are handled via \cite{lineax2023}. All the experiments utilized an Nvidia A100 80GB Graphics Processing Unit(GPU) using double precision which is crucial for achieving high accuracy. Once the manuscript is accepted for publication, the code to reproduce the experiments will be made available as a public Github repository. The first three experiments of this Section use a direct method for computing the solution of $G(\theta, \psi)^\dagger\nabla L(\theta,\psi)$. The experiment described in Section~\ref{sec:MFTG} utilizes the Conjugate Gradient Solver from the Jax Scipy module \cite{jax2018github} to implement the matrix-free formulation detailed in Section~\ref{sec:matrix_free}. We refer to Table ~\ref{tab:optimization_settings} for the optimization settings for each solver for the first three experiments.

\begin{table}[h]
\centering
\begin{tabular}{|l|c|r|}
\hline
\textbf{Method} & \textbf{Number of Iterations}  \\
\hline
GNNG & 5000  \\
BFGS & 5000  \\
Adam & 200000 \\
\hline
\end{tabular}
\caption{Optimization settings for the different solvers used in the first three experiments.}
\label{tab:optimization_settings}
\end{table}

\subsection{Kovasznay Flow}\label{sec:kovasznay}

We consider the two-dimensional steady Navier-Stokes flow as originally described by \cite{kovasznay_1948} with a Reynolds number \(Re=40\). The flow is defined over the computational domain \(\Omega = [-0.5, 1.0] \times [-0.5, 1.5]\). The analytical solutions are given by: 
\begin{align*}
    \begin{split}
    u^*(x, y) &= 1 - e^{\lambda x} \cos(2\pi y), \\
    v^*(x, y) &= \frac{\lambda}{2\pi} e^{\lambda x} \sin(2\pi y), \\
    p^*(x, y) &= \frac{1}{2} (1 - e^{2\lambda x}),
    \end{split}
\end{align*}
where \( \lambda = \frac{1}{2\nu} - \sqrt{\frac{1}{4\nu^2} + 4\pi^2} \), and \( \nu = \frac{1}{\text{Re}} = \frac{1}{40} \).

\begin{figure}[H]
\centering
\includegraphics[width=0.48\textwidth]{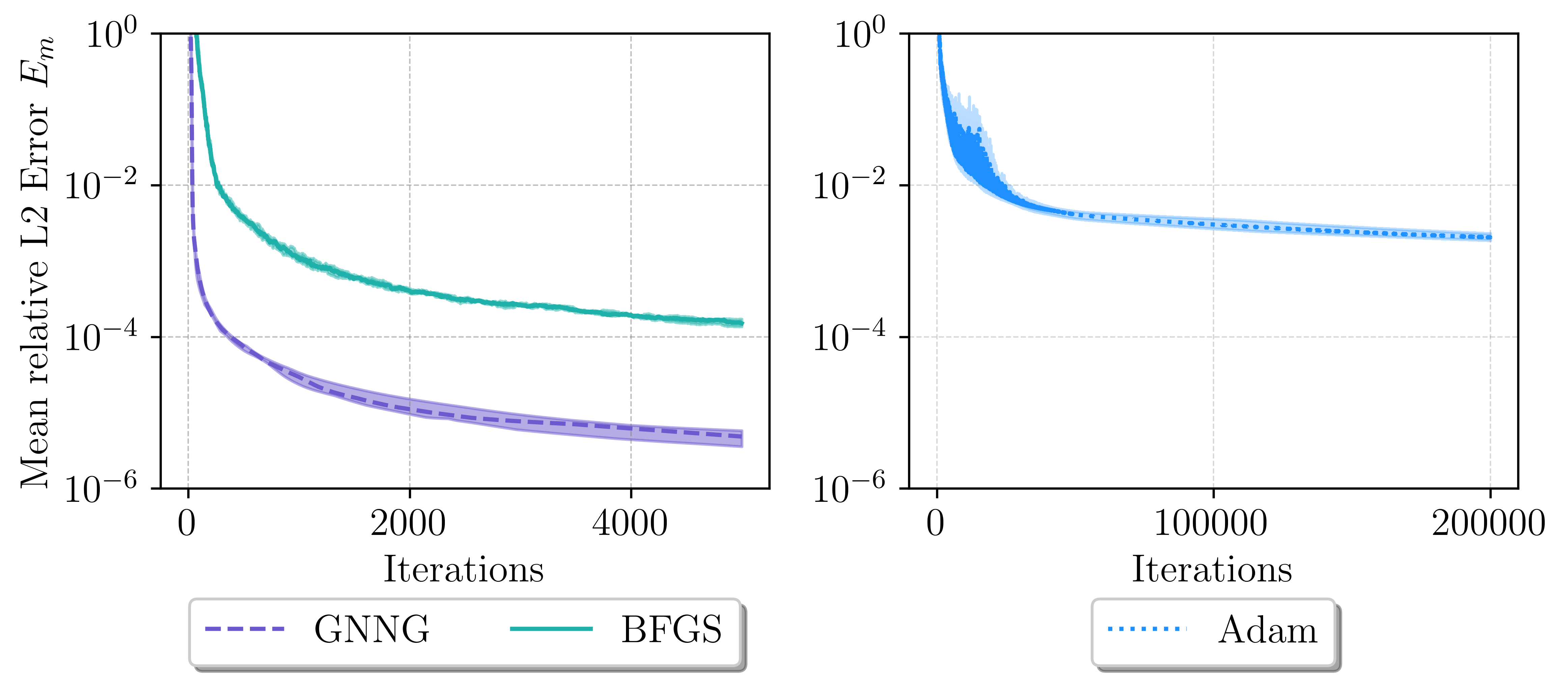}
\caption{Median mean relative \(L^2\) errors \( E_{m} \) during training for the Kovasznay flow. Statistics are computed over 10 different initializations with the shaded area displaying the region between the first and third quartile.}
\label{fig:convergence_comparison_median}
\end{figure}

We select 2601 equidistantly spaced collocation points in the interior of \(\Omega\) and 400 collocation points on the boundary and 26010  equidistantly spaced points for validation. The boundary constraints are imposed in a soft manner and are penalized in the loss function. We employ an architecture with 4 layers of width 50 and conduct the experiments over 10 different Glorot uniform random initializations \cite{Glorot2010UnderstandingTD}.

\begin{table}[h]
    \centering
    \begin{tabular}{|l|c|c|c|}
    \hline
    \textbf{Solver} & \textbf{Min \(E_{m}\) } & \textbf{Max \(E_{m}\) } & \textbf{Median \(E_{m}\) } \\
    \hline
    GNNG & 2.2339e-06 & 7.1622e-06 & 4.8411e-06 \\
    BFGS & 1.0175e-04 & 1.8478e-04 & 1.5057e-04 \\
    Adam & 1.6284e-03 & 2.7729e-03 & 2.0492e-03 \\
    \hline
    \end{tabular}
    \caption{Statistical summary of the mean component-wise relative \(L^2\) errors (\(E_{m}\)) in the Kovasznay Flow experiment. This table presents the minimum, maximum, and median values of \(E_{m}\) across 10 initializations for each solver.}
    \label{tab:max_error}
\end{table}

\begin{table}[h]
    \centering
    \begin{tabular}{|l|c|c|c|c|}
    \hline
    \textbf{Solver} & \textbf{\(u\)} & \textbf{\(v\)} & \textbf{\(p\)}  \\
    \hline
    GNNG & 1.5086e-06 & 4.4716e-06 & 7.2157e-07  \\
    BFGS & 1.7129e-05 & 2.6175e-04 & 2.6375e-05  \\
    Adam & 3.0956e-04 & 4.2589e-03 & 3.1659e-04  \\
    \hline
    \end{tabular}
    \caption{Component-wise relative \(L^2\) errors for \(u\), \(v\), and \(p\) in the Kovasznay flow experiment for the seed with the lowest overall \(E_{m}\) for each solver.}
    \label{tab:best_l2_error}
\end{table}

As reported in Tables \ref{tab:max_error} and \ref{tab:best_l2_error}, and illustrated in Figure~\ref{fig:convergence_comparison_median}, we observe that the first-order optimizer Adam reaches an accuracy plateau at about \(10^{-3}\), despite being allowed a substantially greater number of iterations than the second order optimizers. In contrast, the quasi-Newton method BFGS, although it outperforms Adam, does not reach the close to single-precision accuracy of GNNG, which improves upon BFGS by two orders of magnitude. Comparing to results from the literature confirms the capabilities of GNNG, which improves upon the results for the Kovasznay flow reported in \cite{xiang2021selfadaptive} by up to two orders of magnitude.
To illustrate the geometric interpretation of the optimization dynamics discussed in Section~\ref{sec:geometrical_interpretation}, we visualize the update directions in Figure \ref{fig:kovasznay_pushs} and compare them to the optimal update direction $u_\theta - u^*$. We observe that GNNG yields an excellent visual agreement with $u_\theta - u^*$ unlike the other optimizers. For an extended discussion we refer to Appendix~\ref{sec:visual}. Furthermore, we experimented with ENGD as proposed in \cite{muller2023achieving}. We were however not able to achieve any convergence when used as a stand-alone method. When utilized at a later stage in the optimization process using a different solver first, ENGD yielded comparable results to GNNG. We investigated the situation closer in Appendix~\ref{sec:visual}, visualizing the update directions of ENGD and GNNG. 

\subsection{Beltrami Flow}\label{sec:beltrami}
We consider the unsteady three-dimensional Beltrami flow as originally described by \cite{Ethier1994ExactF3} with a Reynolds number \(\textit{Re}=1\). The flow is defined over the computational domain \(\Omega = [-1, 1] \times [-1, 1] \times [-1, 1]\) and within a time interval of \([0, 1]\). The analytical solutions  are: 
\begin{align*}
    u(x, y, z, t) &= -e^{x} \sin(y + z) + e^{z} \cos(x + y) e^{-t}, \\
        v(x, y, z, t) &= -e^{y} \sin(z + x) + e^{x} \cos(y + z) e^{-t}, \\
        w(x, y, z, t) &= -e^{z} \sin(x + y) + e^{y} \cos(z + x) e^{-t}, \\
\end{align*}
and
\begin{align*}
    p(x, y, z, t) &= -\frac{1}{2} \left[e^{2x} + e^{2y} + e^{2z} \right. \\
        &\quad \left. + 2 \sin(x + y) \cos(z + x)e^{y+z} \right. \\
        &\quad \left. + 2 \sin(y + z) \cos(x + y)e^{z+x} \right. \\
        &\quad \left. + 2 \sin(z + x) \cos(y + z)e^{x+y}\right] e^{-2t}.\\
\end{align*}

\begin{figure}[h]
    \centering
    \includegraphics[width=0.48\textwidth]{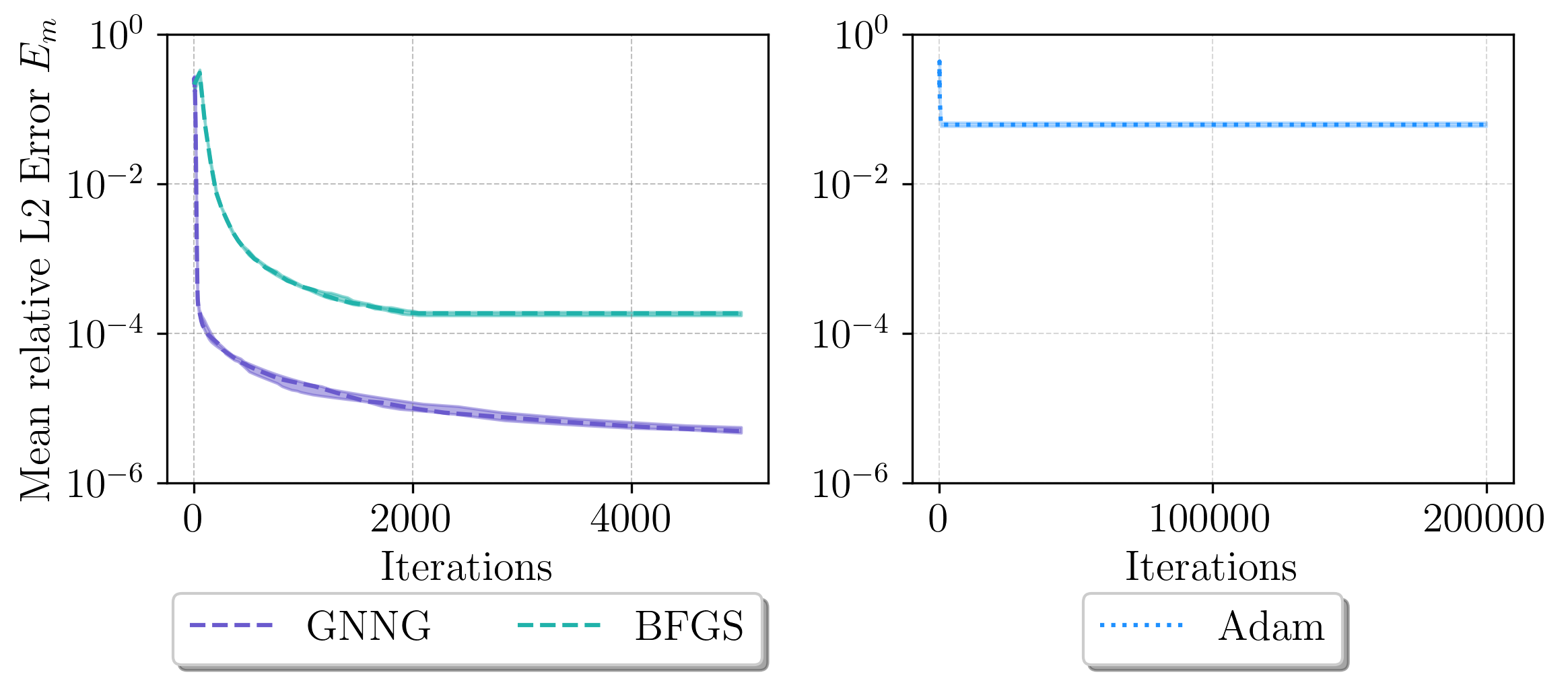}
    \caption{ 
    Median mean relative \(L^2\) errors \( E_{m} \) during training for the Beltrami flow. Statistics are computed over 10 different initializations with the shaded area displaying the region between the first and third quartile.}
    \label{fig:convergence_comparison_medianb}
\end{figure}

\begin{table}[h]
    \centering
    \resizebox{\columnwidth}{!}{
    \begin{tabular}{|l|c|c|c|}
    \hline
    \textbf{Solver} & \textbf{Min \(E_{m}\) } & \textbf{Max \(E_{m}\) } & \textbf{Median \(E_{m}\) } \\
    \hline
    GNNG & 3.4371e-06 & 6.9763e-06 & 4.8992e-06 \\
    BFGS & 2.7475e-04 & 3.6234e-04 & 3.1912e-04 \\
    Adam & 3.9956e-02 & 5.3525e-02 & 4.9906e-02 \\
    \hline
    \end{tabular}}
    \caption{Statistical summary of mean component-wise relative \(L^2\) errors (\(E_{m}\)) in the Beltrami flow experiment. This table presents the minimum, maximum, and median values of \(E_{m}\) across 10 initializations for each solver.}
    \label{tab:max_error_beltrami}
\end{table}

\begin{table}[h]
    \centering
    \resizebox{\columnwidth}{!}{
    \begin{tabular}{|l|c|c|c|c|}
    \hline
    \textbf{Solver} & \textbf{\(u\)} & \textbf{\(v\)} & \textbf{\(w\)} & \textbf{\(p\)}\\
    \hline
    GNNG &  3.4661e-06 & 4.2510e-06 & 4.4721e-06& 1.5591e-06\\
    BFGS & 2.5255e-04 & 3.5580e-04 & 3.7430e-04& 1.1634e-04\\
    Adam & 4.1744e-02 & 5.2890e-02 & 5.5641e-02& 9.5505e-03\\
    \hline
    \end{tabular}
    }
    \caption{Component-wise relative \(L^2\) errors for \(u\), \(v\), \(w\), and \(p\) in the Beltrami flow experiment for the seed with the lowest overall \(E_{m}\) for each solver at \( t = 1 \).}
    \label{tab:best_l2_error_beltrami}
\end{table}

For the training 31 \( \times \) 31 points on each face are used for boundary and initial conditions, while a batch of 10,000 points in the spatio-temporal domain is drawn for the interior collocation points. For this experiment, we present the relative \(L^2\) errors achieved at the final timestep \( t = 1 \). We again report statistics for 10 random parameter initializations.
As observed in Tables~\ref{tab:max_error_beltrami} and \ref{tab:best_l2_error_beltrami}, and illustrated in Figure~\ref{fig:convergence_comparison_medianb}, Adam struggles
and quickly plateaus at an error of \(10^{-2}\) for the final time \( t = 1 \), which agrees with the results reported in \cite{jin2021nsfnets}. 
In \cite{wang2022respecting}, the authors attribute these high errors for unsteady problems to an inherent bias in the PINN formulation as a space-time method. They argue that information from the initial conditions needs to be propagated to the later times and propose a curriculum learning strategy to respect the temporal causality. Note that BFGS mitigates this phenomenon and GNNG is not affected by it at all, reaching a relative error as low as \(10^{-6}\) at the final time. The efficiency of GNNG stems from its interpretation as a natural gradient method. The metric $g$ defined in Section~\ref{sec:geometrical_interpretation} takes into account the whole time horizon at once and thus respects temporal causality.

\subsection{Taylor-Green Vortex}\label{sec:Taylor_Green}
The Taylor-Green Vortex, as originally analyzed in \cite{1937RSPSA.158..499T}, is considered within the computational domain \(\Omega = [0, 2\pi] \times [0, 2\pi]\) and the time interval \([0, 10]\). The velocities and pressure are given by
\begin{align*}
    \begin{split}
        u(x, y, t) &= \sin(x) \cos(y) F(t), \\
        v(x, y, t) &= -\cos(x) \sin(y) F(t), \\
        p(x, y, t) &= \frac{1}{4} (\cos(2x) + \cos(2y)) F^2(t),
    \end{split}
\end{align*}
where \( F(t) = e^{-2\nu t} \) and \( \nu = \frac{1}{\text{Re}} = \frac{1}{500} \), assuming the fluid density \( \rho = 1 \).

For the training data, a batch of 8,000 points in the spatio-temporal domain is used for the interior collocation points. Again, we present the relative \(L^2\) errors achieved at the final time \( t = 10 \) and conduct the experiments over 10 different initializations performed according to the Glorot uniform scheme. While maintaining the architecture of 4 layers of width 50, we implement several architectural transformations to the neural network for this experiment:
\begin{itemize}
    \item \textbf{Divergence-Free Condition:} We use the ansatz described in \cite{richter2022neural} to generate a model \( N(t,x; \theta_u) \) that has divergence-free output.

    \item \textbf{Exact Imposition of Initial Conditions:} In accordance with the approach described in \cite{lu2021physicsinformed}, we transform the output so that it exactly imposes initial conditions. The neural network output is modified as \( \hat{u}(t,x; \theta_u) = g(x) + \ell(t)N(t,x; \theta_u) \), where \( N(t,x; \theta_u) \) is the network output, \( \ell(t) \) is zero at the initial condition \( t = 0 \) and \( g(x) \) is the initial condition to be imposed. This does not affect the divergence-free condition, as the initial condition \( g(x) \) is divergence-free and \( \ell(t) \) is independent of spatial variables.

    \item \textbf{Periodic Boundary Conditions:} Following the approach in \cite{Dong_2021}, we replace the input in each spatial dimension \( x_j \) by two terms of the basis functions of the Fourier series : \( \cos\left(2\pi x_j/P\right) \) and \( \sin\left(2\pi x_j/P\right) \) . This effectively enforces periodicity in the model.
\end{itemize}

These adaptions reduce the loss function to the momentum equations. As reported in Table~\ref{tab:TGmax_error} and Table~\ref{tab:TGbest_l2_error}, and illustrated in the Figure~\ref{fig:convergence_comparison_mediant}, we observe that using hard-imposed initial and divergence conditions greatly improves the accuracy. Adam reaches an accuracy in the order of \(10^{-4}\) and BFGS reaches an accuracy in the order of \(10^{-5}\). GNNG outperforms both Adam and BFGS, by three and two orders of magnitudes respectively, reaching relative \(L^2\) errors as low as \(10^{-7}\).

\begin{figure}[H]
\centering
\includegraphics[width=0.48\textwidth]{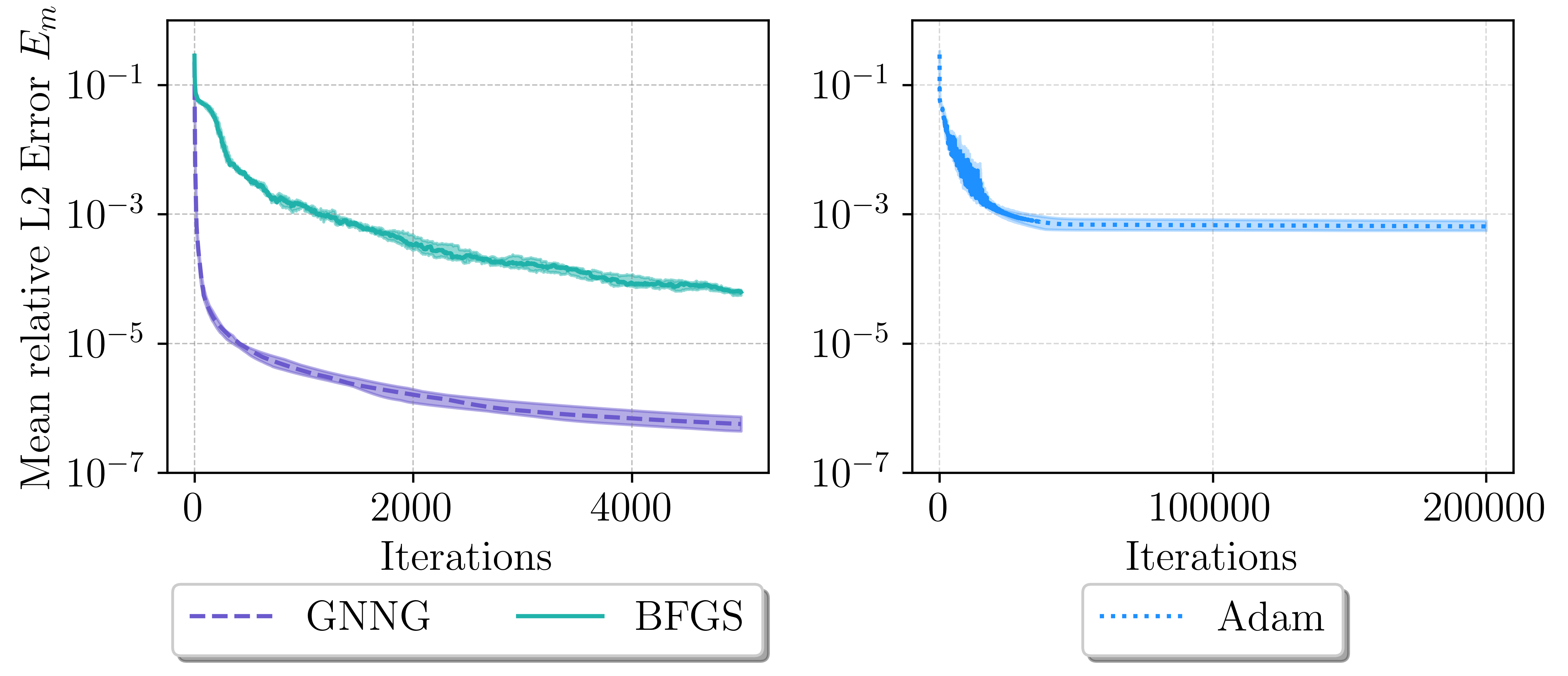}
\caption{Median mean relative \(L^2\) errors \( E_{m} \) during training for the Taylor-Green vortex. Statistics are computed over 10 different initializations with the shaded area displaying the region between the first and third quartile.}
\label{fig:convergence_comparison_mediant}
\end{figure}

    \begin{table}[h!]
    \centering
    \begin{tabular}{|l|c|c|c|}
    \hline
    \textbf{Solver} & \textbf{Min \(E_{m}\) } & \textbf{Max \(E_{m}\) } & \textbf{Median \(E_{m}\) } \\
    \hline
    GNNG & 4.2921e-07 & 8.2343e-07 & 5.6828e-07 \\
BFGS & 4.1266e-05 & 8.7792e-05 & 6.2360e-05 \\
Adam & 4.6008e-04 & 9.4373e-04 & 6.4602e-04 \\
\hline
    \end{tabular}
    \caption{Statistical summary of the mean component-wise relative \(L^2\) errors (\(E_{m}\)) for different solvers in the Taylor-Green vortex experiment evaluated at the final timestep \(t\)=10. The table reports the minimum, maximum, and median of \(E_{m}\) values obtained from 10 initializations for each solver.}
\label{tab:TGmax_error}
    \end{table}

\begin{table}[h!]
\centering
\begin{tabular}{|l|c|c|c|}
\hline
\textbf{Solver} & \textbf{\(u\)} & \textbf{\(v\)} & \textbf{\(p\)}  \\
\hline
    GNNG &  4.1925e-07 & 1.9107e-07 & 6.7731e-07  \\
BFGS & 2.6636e-05 & 3.8090e-05 & 5.9073e-05 \\
Adam & 4.5733e-04 & 4.5708e-04 & 4.6582e-04 \\ 
\hline
\end{tabular}
\caption{Component-wise relative \(L^2\) errors for \(u\), \(v\), and \(p\) in the Taylor-Green vortex experiment for the seed with the lowest overall \(E_{m}\) for each solver evaluated at the final timestep \(t\)=10.}
\label{tab:TGbest_l2_error}
\end{table}

\section*{Acknowledgements}
AJ was supported by a fellowship from Leonardo S.p.A.

\section{Conclusion}
We discretized Gauss-Newton's method in function space in the tangent space of a neural network ansatz and demonstrated its excellent performance on a number of Navier-Stokes benchmark problems. Exploiting its connection to Gauss-Newton in parameter space, we successfully demonstrated scalability to large neural network ansatz functions.

\bibliography{example_paper}
\bibliographystyle{icml2024}

\newpage
\appendix
\onecolumn
\section{Reference Description of the Algorithm}\label{sec:concise_description}
In this Section we provide a concise description of the algorithm, mainly for reference purposes. We include the explicit formula for the matrix $G$ of GNNG, compare to \eqref{eq:generic_second_order}. Here, we use two neural networks $u_\theta$ and $p_\psi$ for the velocity and the pressure respectively with parameters $\theta\in\Theta=\mathbb R^{p_\Theta}$ and $\psi\in\Psi=\mathbb R^{p_{\Psi}}$. By $L$ we denote the loss function for a PINN formulation of the Navier-Stokes equations, see also \eqref{eq:loss_navier_stokes_pinn}. The GNNG method is summarized in Algorithm \ref{algo:alg}.
\begin{algorithm}
    \caption{Gauss-Newton Natural Gradient with Line Search}\label{alg:GNNG}
    \begin{algorithmic}
    \STATE {\bfseries Input:} initial parameters $\theta_0\in \Theta, \psi_0\in\Psi$,  
    $N_{max}$ 
    \FOR{$k=1, \dots, N_{max}$}   
    \STATE Compute $\nabla L(\theta, \psi)\in\mathbb R^{p_\Theta}\times R^{p_\Psi}$ 
    \STATE Assemble $G(\theta, \psi)$ 
    \STATE $\nabla^G L(\theta, \psi) \gets G^\dagger(\theta)\nabla L(\theta, \psi)$ 
    \STATE $\eta^* \gets \arg\min_{\eta\in[0,1]} \,L( (\theta,\psi) - \eta \nabla^G L(\theta, \psi) )$ 
    \STATE $(\theta_k, \psi_k) = (\theta_{k-1}, \psi_{k-1}) - \eta^* \nabla^G L(\theta, \psi)$ 
    \ENDFOR
    \end{algorithmic}
    \label{algo:alg}
\end{algorithm}
The matrix $G(\theta, \psi)$ has block structure
\begin{align*}
    G(\theta, \psi)
    =
    \begin{pmatrix}
        A & B \\
        B^T & C
    \end{pmatrix}.
\end{align*}
For the case of the stationary Navier-Stokes equations, the blocks are given by
\begin{align*}
    A_{ij}
    &=
    (-\nu \Delta \partial_{\theta_j}u_\theta + (\partial_{\theta_j}u_\theta \cdot \nabla)u_\theta + (u_\theta\cdot\nabla)\partial_{\theta_j}u_\theta
    , 
    -\nu \Delta \partial_{\theta_i}u_\theta + (\partial_{\theta_i}u_\theta \cdot \nabla)u_\theta + (u_\theta\cdot\nabla)\partial_{\theta_i}u_\theta )_{L^2(\Omega)}
    \\
    &+
    (\operatorname{div}(\partial_{\theta_j}u_\theta), \operatorname{div}(\partial_{\theta_i}u_\theta))_{L^2(\Omega)}
    +
    (\partial_{\theta_j}u_\theta, \partial_{\theta_i}u_\theta)_{L^2(\Omega)},
    \\
    B_{ij}
    &=
    (-\nu \Delta \partial_{\theta_j}u_\theta + (\partial_{\theta_j}u_\theta \cdot \nabla)u_\theta + (u_\theta\cdot\nabla)\partial_{\theta_j}u_\theta
    ,
    \nabla \partial_{\psi_i}p_\psi)_{L^2(\Omega)}
    \\
    C_{ij}
    &=
    (\nabla \partial_{\psi_j}p_\psi, \nabla \partial_{\psi_i}p_\psi)_{L^2(\Omega)}.
\end{align*}
For numerical stability we will employ a damping of the matrix $G$ by an additive correction of $G$ via a scaled identity, i.e., we use 
\[
    G(\theta,\psi) + \min(10^{-5}, L(\theta,\psi))\operatorname{Id}.
\]
Damping is typical in the natural gradient community \cite{martens2015optimizing, martens2020new} and our specific choice performed well for all experiments.

\section{Proof of the Projection Theorem}\label{sec:appendix_math}
In this Section, we provide the missing details for the proof of Theorem \ref{thm:update_direction}. The analysis will be carried out assuming that the neural network ansatz satisfies boundary, initial, and divergence constraints exactly, compare also to the numerical example \ref{sec:Taylor_Green}.

To complete the proof we need to show that the Riemannian metric
\begin{equation}\label{appendix:metric}
    g(u_k, p_k)((u,p), (v,q))
    =
    (DR(u_k, p_k)(u, p), DR(u_k, p_k)(v, q))_{\mathcal H}
\end{equation}
is positive definite. To this end, we will work in the simplified setting of enforcing the boundary, divergence and initial conditions directly in the neural network ansatz in the functional setting. Furthermore, we will assume that the initial condition $u_0$ and the boundary values $g$ vanish.\footnote{For non-homogeneous initial boundary and initial conditions we can shift the problem.} We introduce now the functional setting where we follow \cite{hinze2000optimal}. Let $I=[0,T]$ be a time interval and let $\Omega\subset\mathbb R^2$ be a domain with $C^2$ boundary. Further, we set $\Omega_T = I\times \Omega$. Consider the spaces 
\begin{align*}
    V 
    &=
    \operatorname{cl}_{H^1(\Omega)}\{ u\in C^\infty_0(\Omega)^2\mid \operatorname{div}(u)=0 \},
    \\
    H^{2,1}(\Omega_T) 
    &=
    L^2(I, H^2(\Omega)\cap V)\cap H^1(I, L^2(\Omega)),
\end{align*}
Here $V$ is the closure with respect to the $H^1(\Omega)$ norm of the smooth functions with compact support and vanishing divergence. The space $H^{2,1}(\Omega_T)$ is the maximal parabolic regularity space. Under the above assumptions, we have that the metric in \eqref{appendix:metric} reduces to the contribution of the momentum equations. To guarantee its definiteness we need to analyze the term
\begin{align*}
    g(u_k,p_k)((u,p), (u,p))
    &=
    \| DR(u_k, p_k)((u,p), (u,p)) \|^2_{L^2(\Omega_T)}
    \\
    &=
    \| \partial_t u - \nu\Delta u + (u_k\cdot \nabla)u + (u\cdot\nabla)u_k + \nabla p \|^2_{L^2(\Omega_T)}.
\end{align*}

\begin{proposition}
    Assume we are in the setting outlined above. Then, for all $u\in H^{2,1}(\Omega_T)$ with $u(0)=0$ and $p\in H^1(\Omega)\cap L^2_0(\Omega)$ it holds
    \begin{equation*}
        \| u \|^2_{H^{2,1}(\Omega_T)} + \|\nabla p\|^2_{L^2(\Omega_T)}
        \lesssim
        \|
        \partial_tu - \nu\Delta u + (u_k\cdot \nabla)u + (u\cdot\nabla)u_k + \nabla p 
        \|^2_{L^2(\Omega_T)}
        \lesssim
        \| u \|^2_{H^{2,1}(\Omega_T)} + \|\nabla p\|^2_{L^2(\Omega_T)}.
    \end{equation*}
    In particular, $g(u_k,p_k)$ is positive definite.
\end{proposition}
\begin{proof}
    This lower estimate follows essentially from Proposition 2.1 item vi. in \cite{hinze2000optimal}. To recover the pressure that is absent in this reference, we note that we can recover it extending the test functions in the variational formulation of $DR(u_k,p_k)$ from $V$ to all of $H^1_0(\Omega)$. The remaining estimate is a straight-forward computation.
\end{proof}

\section{Visual Comparison of the Update Directions}\label{sec:visual}
As commented earlier, Newton's method in function space, proposed as ENGD in \cite{muller2023achieving} does not yield satisfactory results for the Navier-Stokes equations when applied early in the training process. To illustrate this fact we compare the update directions for the $u$ component of the velocity in the Kovasznay flow example of Section \ref{sec:kovasznay} for different iterations in the training process. We note that in the beginning of the training process, see Figure \ref{fig:kovasznay_pushs_extended20}, the update direction proposed by ENGD does not match the error -- which is the optimal update direction -- closely. This is in contrast to the update of the Gauss-Newton method which visually yields an almost perfect fit. Later in the training process, in this example after 70 iterations, when the loss function value has decreased to around $2\mathrm{e}{-4}$, the update directions of Gauss-Newton and ENGD agree, compare to Figure \ref{fig:kovasznay_pushs_extended70}.

\begin{figure}[h]
    \centering
    \includegraphics[width=0.8\linewidth]{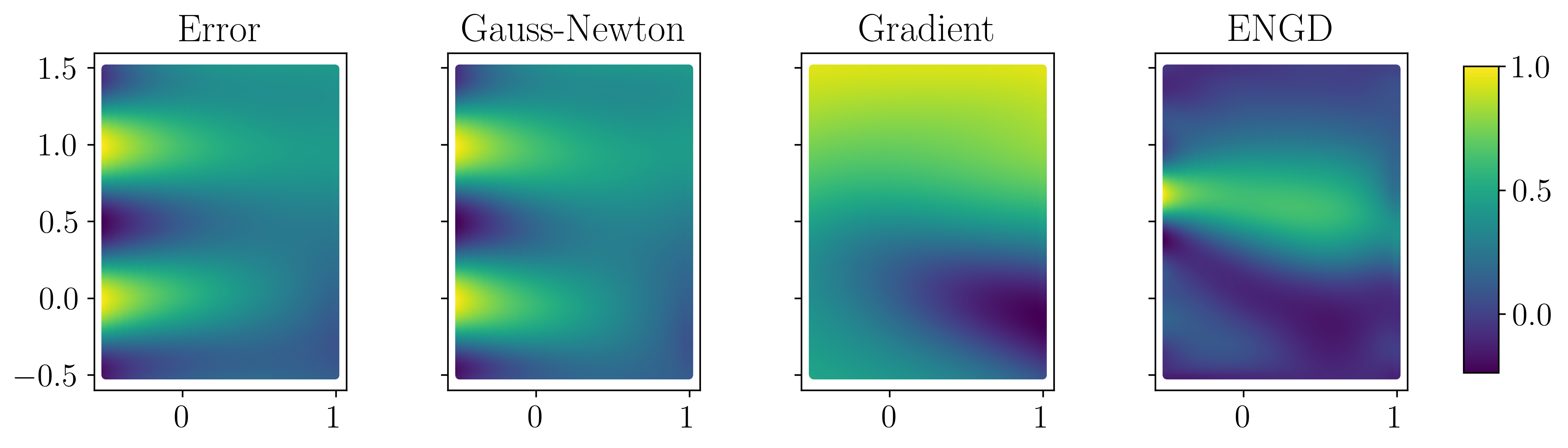}
    \caption{Visualization of the update directions of different optimizers in the example of the Kovasznay flow at the 20th iteration of a Gauss-Newton solve. Note that the error is the optimal update direction. The first component of the velocity is shown and all plots are normed to lie in $[-1,1]$.}
    \label{fig:kovasznay_pushs_extended20}
\end{figure}

\begin{figure}[h]
    \centering
    \includegraphics[width=0.8\linewidth]{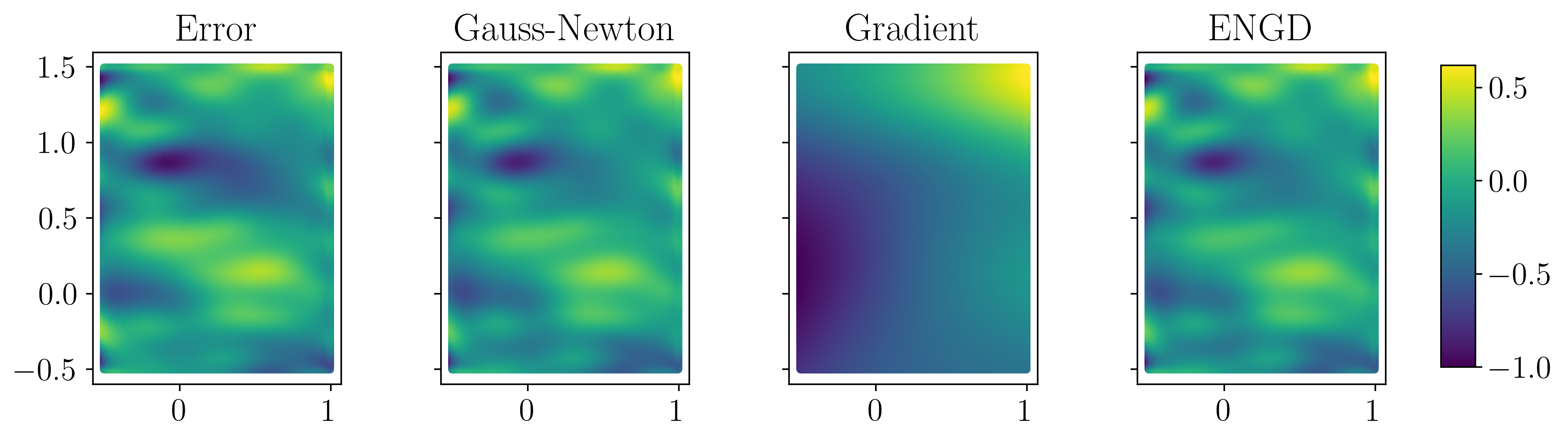}
    \caption{Visualization of the update directions of different optimizers in the example of the Kovasznay flow at the 70th iteration of a Gauss-Newton solve. Note that the error is the optimal update direction. The first component of the velocity is shown and all plots are normed to lie in $[-1,1]$.}
    \label{fig:kovasznay_pushs_extended70}
\end{figure}

\newpage

\section{Matrix-Free Taylor-Green}\label{sec:MFTG}
We revisit the Taylor-Green vortex as presented in Section~\ref{sec:Taylor_Green} with a modified setup to demonstrate the computational feasibility of employing large-networks using the matrix-free formulation described in \ref{sec:matrix_free}. The neural network architecture is expanded to 10 layers of width 100, which roughly corresponds to 91k parameters. Solving the linear system using the direct method would require storing a Gramian matrix of size 62.79 gigabytes per collocation point. BFGS, which stores a dense $n\times n$  approximation of the inverse Hessian, where $n$ is the number of parameters of the network, is not usable in this case and we instead benchmark against limited-memory BFGS (L-BFGS), see for instance \cite{Liu_Nocedal_1989}. The matrix-free system is solved using the conjugate gradient method, with tolerance $10^{-5}$. Reducing the tolerance leads to increased accuracy at the price of increased computational time.

Furthermore, unlike the experiment in Section~\ref{sec:Taylor_Green} we impose the constrains in a soft-manner, maintaining only the periodic boundary conditions as a hard constraint. We refer to  Table~\ref{tab:optimization_settingsTG} for the optimization settings for this problem. We conduct the experiment for 5 different initializations. The rest of the setup remains unchanged.

As reported in Tables \ref{tab:TGEMF} and \ref{tab:TGbest_l2_error}, and illustrated in Figure~\ref{fig:convergence_comparison_mediantTGMF}, and similarly to the Beltrami flow experiment described in Section~\ref{sec:beltrami}, Adam struggles to reach satisfying accuracy with soft constraints, and only achieves \(L^2\) errors of the order of $10^{-1}$. We also note the loss in accuracy comparing L-BFGS to BFGS, reaching errors in the order of $10^{-2}$. More notably, despite being ran for substantially less iterations, GNNG in the matrix-form is able to produce highly accurate solutions. We achieve relative \(L^2\) errors of the order  of $10^{-5}$. It is also important to highlight that in this case, the accuracy of GNNG is primarily constrained by the tolerance level set for the conjugate gradient.

\begin{table}[h]
    \centering
    \begin{tabular}{|l|c|r|}
    \hline
    \textbf{Method} & \textbf{Number of Iterations}  \\
    \hline
    GNNG & 1000  \\
    L-BFGS & 5000  \\
    Adam & 100000 \\
    \hline
    \end{tabular}
    \caption{Optimization settings for the different solvers used in the experiments.}
    \label{tab:optimization_settingsTG}
\end{table}

\begin{figure}[H]
    \centering
    \includegraphics[width=0.48\textwidth]{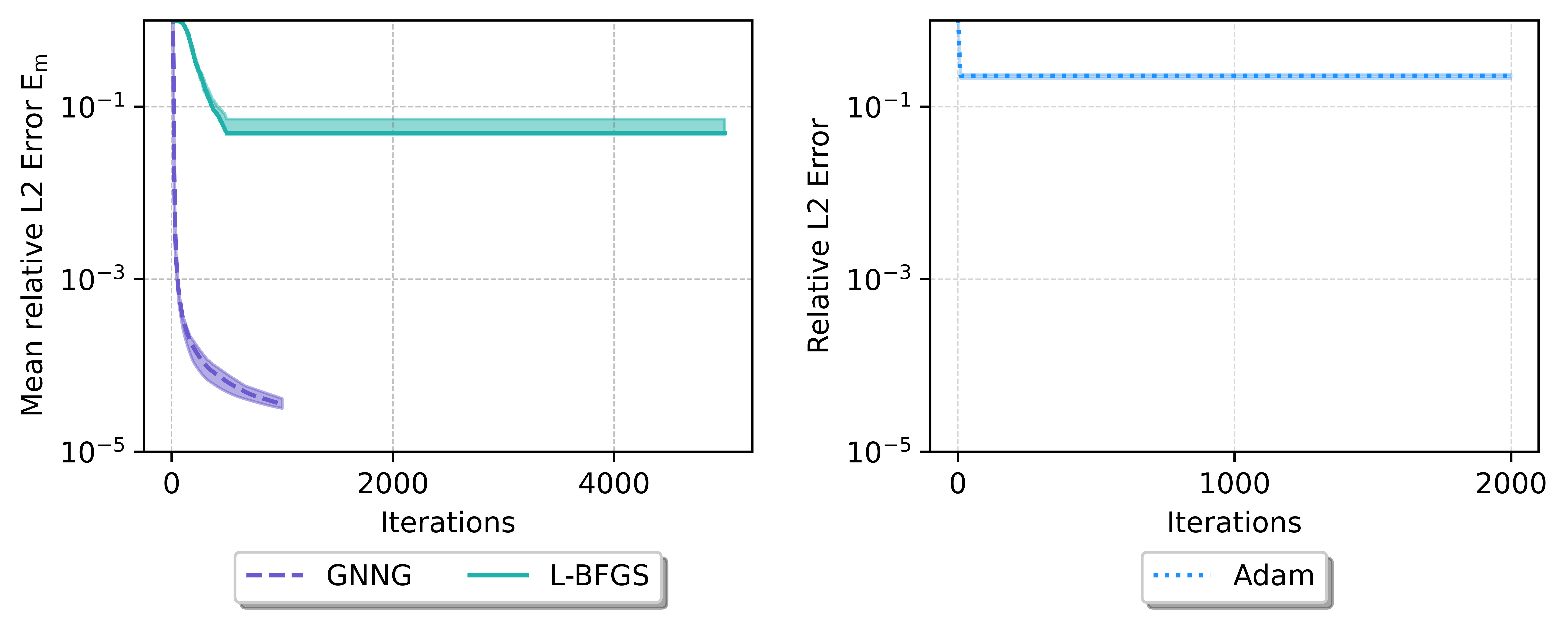}
    \caption{Median mean relative \(L^2\) errors \( E_{m} \) throughout the training process for the Matrix-free Taylor-Green vortex with soft boundary conditions over 5 different initializations. The shaded area displays the region between the first and third quartile.}
    \label{fig:convergence_comparison_mediantTGMF}
\end{figure}

\begin{table}[H]
    \centering
    \begin{tabular}{|l|c|c|c|}
    \hline
    \textbf{Solver} & \textbf{Min \(E_{m}\) } & \textbf{Max \(E_{m}\) } & \textbf{Median \(E_{m}\) } \\
    \hline
    GNNG & 1.8961e-05 & 4.5460e-05 & 3.5657e-05 \\
    L-BFGS &  3.3758e-02 & 7.5131e-02 & 4.9647e-02 \\
    Adam & 1.9082e-01 & 2.7766e-01 & 2.2857e-01 \\
    \hline
    \end{tabular}
    \caption{Statistical summary of the mean component-wise relative \(L^2\) errors (\(E_{m}\)) in the matrix-free Taylor Green vortex experiment. This table presents the minimum, maximum, and median values of \(E_{m}\) across 5 initializations for each solver.}
    \label{tab:TGEMF}
\end{table}

\begin{table}[h!]
    \centering
    \begin{tabular}{|l|c|c|c|}
    \hline
    \textbf{Solver} & \textbf{\(u\)} & \textbf{\(v\)} & \textbf{\(p\)}  \\
    \hline
    GNNG &  1.8827e-05 & 1.6894e-05 & 2.1160e-05  \\
    L-BFGS & 2.8716e-02 & 2.8652e-02 & 4.3906e-02 \\
    Adam & 1.5338e-01 & 1.4830e-01 & 2.7078e-01 \\ 
    \hline
    \end{tabular}
    \caption{Component-wise relative \(L^2\) errors for \(u\), \(v\), and \(p\) in the Matrix Free Taylor-Green vortex experiment for the seed with the lowest overall \(E_{m}\) for each solver evaluated at the final timestep \(t\)=10.}
\end{table}

\section{ Additional Resources for the Experiments}\label{sec:Additional}

\subsection{Loss Functions}\label{sec:Losses}
We present the loss functions for the experiments detailed in Section~\ref{sec:Experiments}, Each plot depicts the median loss over 10 iterations, with the shaded area representing the interquartile range.

\begin{figure}[H]
\centering
\includegraphics[width=0.48\textwidth]{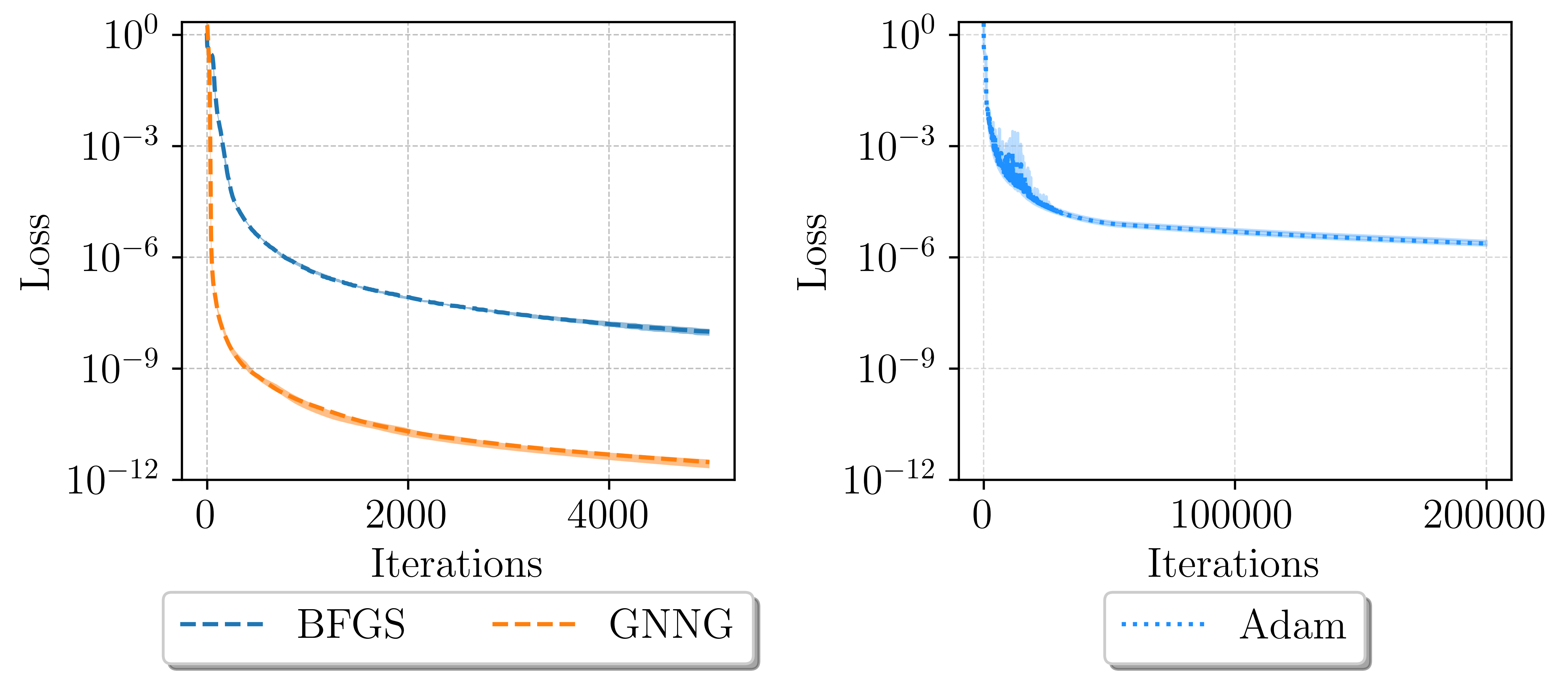}
\caption{Median loss function value during the optimization process for the Kovasznay flow experiment. The statistics are computed over 10 different initializations with the shaded area displaying the region between the first and third quartile.}
\label{fig:kovasznay_loss}
\end{figure}

\begin{figure}[H]
\centering
\includegraphics[width=0.48\textwidth]{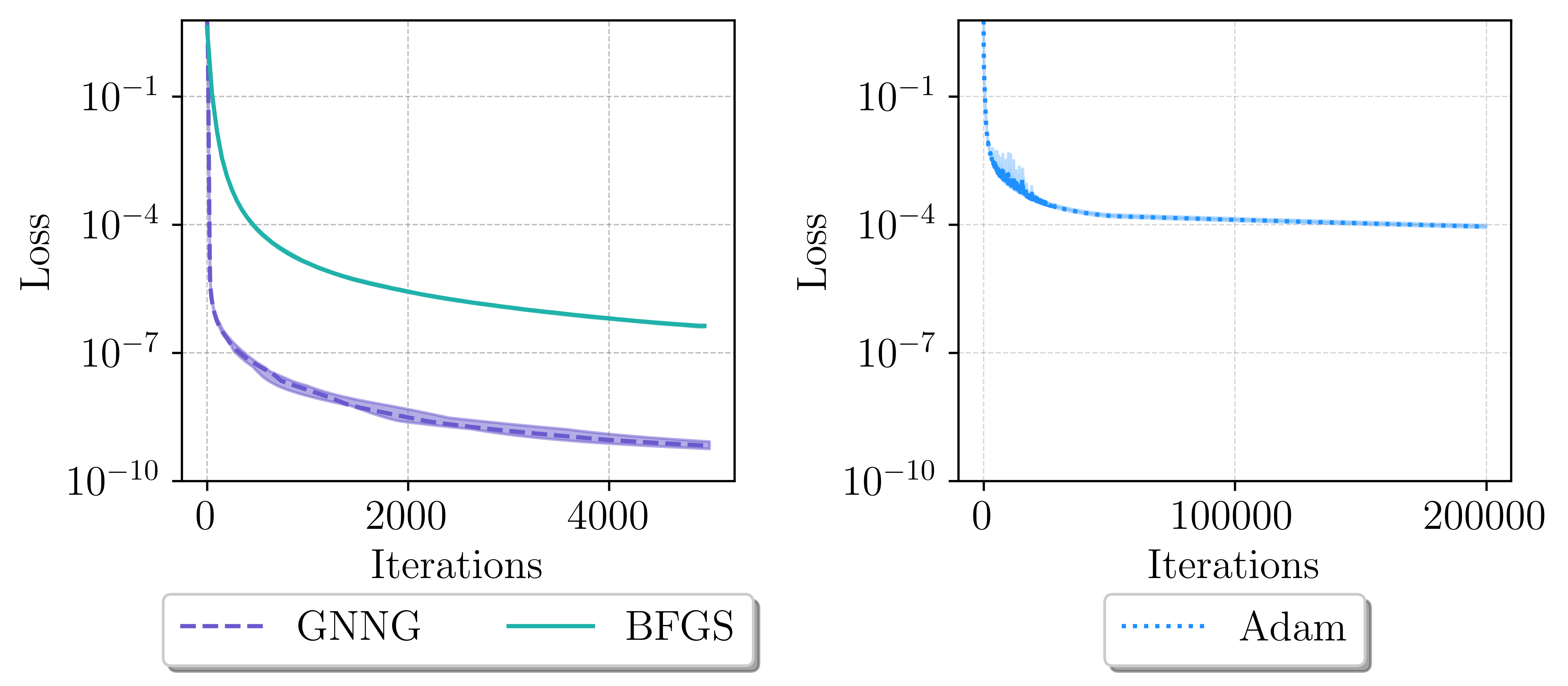}
\caption{Median loss function value during the optimization process for the Beltrami flow experiment. The statistics are computed over 10 different initializations with the shaded area displaying the region between the first and third quartile.}
\label{fig:los_bs}
\end{figure}

\begin{figure}[H]
\centering
\includegraphics[width=0.48\textwidth]{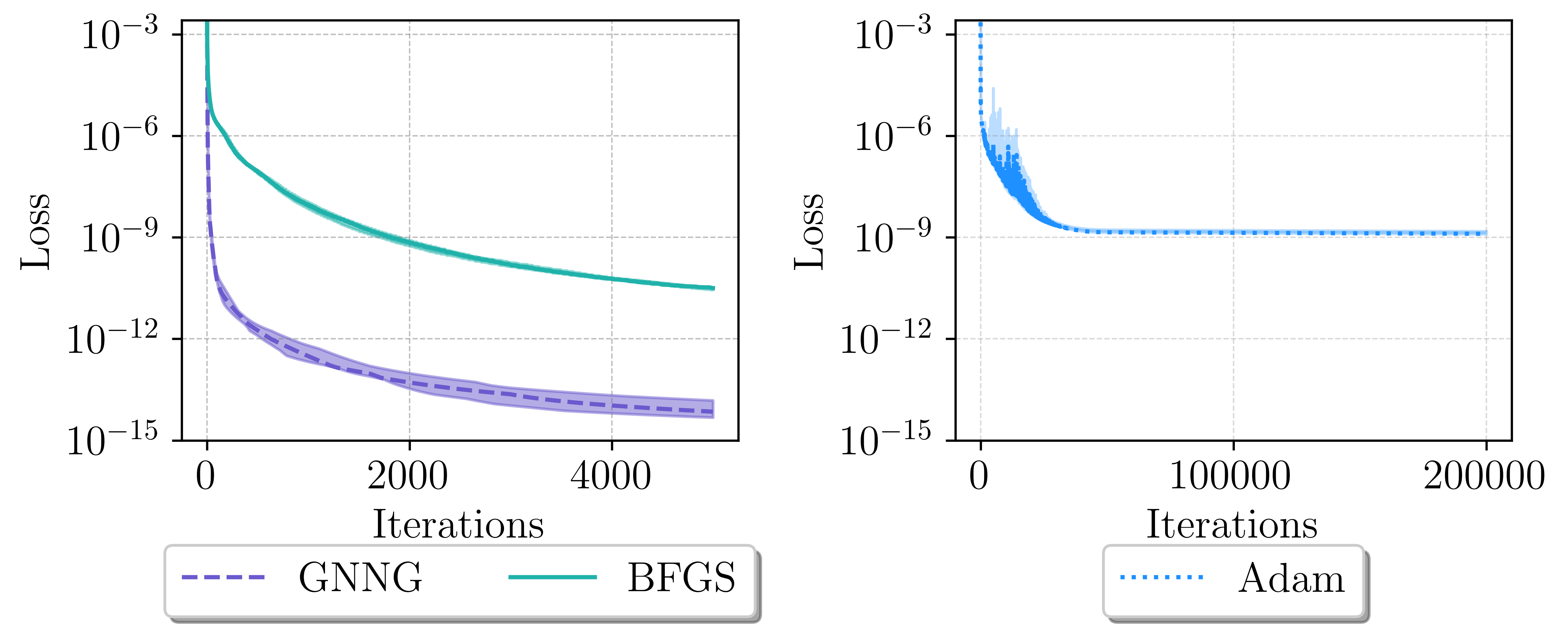}
\caption{Median loss function value during the optimization process for the Taylor-Green vortex with hard-imposed boundary conditions. The statistics are computed over 10 different initializations with the shaded area displaying the region between the first and third quartile.}
\label{fig:taylor_green_loss}
\end{figure}

\begin{figure}[H]
\centering
\includegraphics[width=0.48\textwidth]{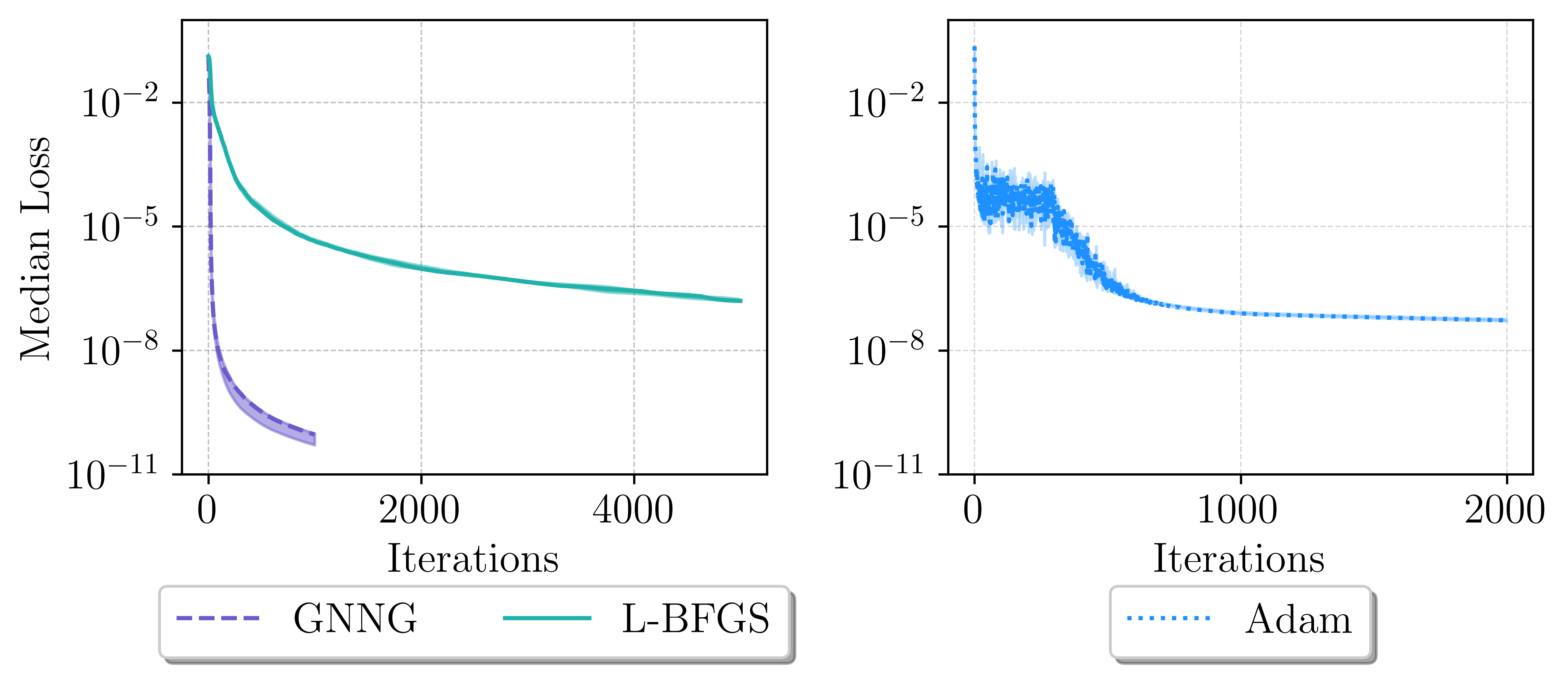}
\caption{Median loss function value during the optimization process for the matrix-free Taylor-Green vortex with soft boundary conditions. The statistics are computed over 5 different initializations with the shaded area displaying the region between the first and third quartile.}
\label{fig:loss_cmg}
\end{figure}

\end{document}